\documentclass[10pt,a4paper,twoside]{amsart}

\usepackage[utf8]{inputenc} 
\usepackage[UKenglish]{babel} 
\usepackage[T1]{fontenc} 
\usepackage{xcolor} 
\usepackage{microtype, fancyhdr, lmodern}
\usepackage[a4paper, hmarginratio=1:1]{geometry}

\numberwithin{equation}{section}	
\usepackage{faktor}
\allowdisplaybreaks

\usepackage{url, enumitem}
\usepackage[noadjust]{cite}

\usepackage{caption, graphicx}
\graphicspath{{Figures/}}

\usepackage[colorlinks=true, linkcolor=blue]{hyperref}
\usepackage[noabbrev]{cleveref}

\theoremstyle{plain}
\newtheorem{theorem}{Theorem}[section]
\newtheorem{lemma}[theorem]{Lemma}

\newtheorem{corollary}[theorem]{Corollary}

\theoremstyle{definition}

\theoremstyle{remark}
\newtheorem{remark}[theorem]{Remark}

\begin{document}



\title[On wen knots]
{On wen knots}

\date{\today}


\author{Celeste Damiani}
\address{Fondazione Istituto Italiano di Tecnologia (IIT), 
Via Morego 30, 16163 Genova, Italy.}
\email{celeste.damiani@iit.it}

\author{Shin Satoh}
\address{Department of Mathematics, 
Kobe University, 
Rokkodai-cho 1-1, Nada-ku, 
Kobe 657-8501, Japan}
\email{shin@math.kobe-u.ac.jp}

\thanks{During the writing of this paper C.D. was initially supported by JSPS KAKENHI Grant Number JP16F1679 and a JSPS Postdoctoral Fellowship For Foreign Researchers.  C.D is a member of GNSAGA of INdAM.
S. S. is partially supported by 
JSPS Grants-in-Aid for Scientific Research (C), 
19K03466.}


\begin{abstract} 
We introduce the notion of wen knots, 
and prove that the set of wen knots 
is a proper subset of the set of extended welded knots. 
Furthermore we prove that 
the complementary subset consists of 
welded knots up to horizontal mirror reflections. This allow us to characterise completely extended welded knots by the parity of their number of wens, that we can always reduce to 0 or~1.
\end{abstract}

\maketitle

\section{Introduction}\label{sec1}

Welded knots are an extension of classical knots
in the $3$-sphere \cite{Fenn-Rimanyi-Rourke97, Kau3}, 
and extended welded knots are a further extension of welded knots
introduced in~\cite{D1}. 
Extended welded knots are motivated by the connection between welded knots and ribbon torus-links, which are oriented tori in $S^4$ that bound ribbon tori (immersed solid tori in $S^4$ with a $3$-dimensional orientation on themselves and a 1-dimensional orientation of their core, whose singularity sets consist of a finite number of ribbon disks that act as filling), up to ambient isotopy. Said connection is given by the Tube map, introduced in~\cite{Satoh:2000}, building on the work of Yajima~\cite{Yajima:62}, who defined a map that ``inflates'' classical knots into ribbon torus-knots, hence the name. In~\cite{Satoh:2000} the second author proves that the map is surjective, that the welded combinatorial knot group corresponds to the fundamental group of the complement in $S^4$ of the image, and that the map commutes with the operation of orientation reversal. However, the Tube map is not injective, and its kernel is not understood yet. The Tube map has also been carefully studied in~\cite{Audoux:2016}, where the interpretation of the invariance of the  map under generalised Reidemeister moves is given in terms of local filling changes.
There are hints that extended welded links might be suitable candidates to be a diagrammatic representation of ribbon torus-links, overcoming the non-injectivity of the Tube map defined on (non-extended) welded links: for instance, extended welded knots are equivalent to their horizontal mirror images~\cite{D2}, and their properties are instrumental in formalising a partial version of Markov's theorem for ribbon torus links~\cite{D3}.

\bigskip

In this paper, we will answer to 
the question: what is an extended welded knot?
For an integer $n\geq 0$, 
let $\mathcal{D}_n$ be the set of oriented virtual knot diagrams 
with $n$ dots called \emph{wens}. 
Consider the following local deformations as shown in
Figures~\ref{F:welded_moves} and \ref{fig101}: 
\begin{itemize}
\item
Classical Reidemeister moves \textrm{R1}--\textrm{R3}. 
\item
Virtual Reidemeister moves \textrm{R4}--\textrm{R7}. 
\item
An upper forbidden move \textrm{R8}.
\item
Wen moves \textrm{W1}--\textrm{W4}.
\end{itemize}
Then \textit{welded knots}, \textit{wen knots}, and 
\textit{extended welded knots} are defined by 
\begin{itemize}
\item
$\{\mbox{welded knots}\}=\mathcal{D}_0/(\mbox{R1--R8})$, 
\item
$\{\mbox{wen knots}\}=
\mathcal{D}_1/(\mbox{R1--R8, W1--W3})$, and 
\item
$\{\mbox{extended welded knots}\}=
\left(\bigsqcup_{n=0}^{\infty}\mathcal{D}_n\right)/(\mbox{R1--R8, W1--W4})$. 
\end{itemize}
That is, 
a welded knot is an equivalence class of virtual knot diagrams 
without wens under Reidemeister moves \textrm{R1}--\textrm{R8}, 
a wen knot is that of virtual knot diagrams 
with a single wen under R1--R8 and wen moves W1--W3 with the exception of W4, 
and an extended welded knot is that of virtual knot diagrams 
with a finite number of wens under \textrm{R1}--\textrm{R8} 
and \textrm{W1}--\textrm{W4}. 
We remark that in~\cite{D1, BKLPY, D2} an angled mark is used 
to indicate a wen instead of a dot.

\begin{figure}[htb]
\centering
\includegraphics[scale=.75]{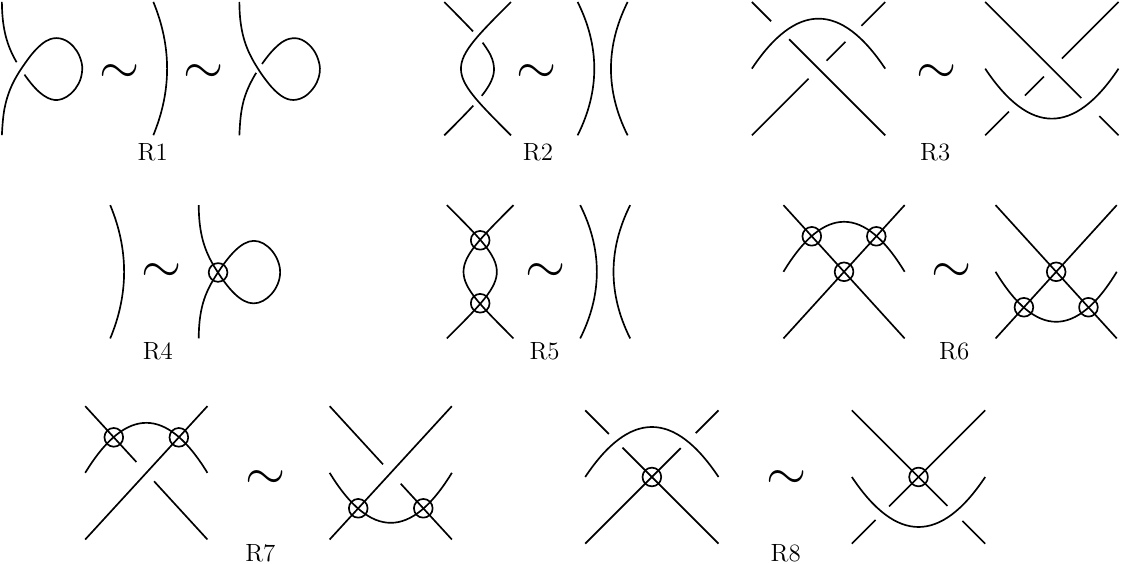}
\caption{Reidemeister moves R1--R8.}
\label{F:welded_moves}
\end{figure}

\begin{figure}[htb]
\centering
\includegraphics[scale=.75]{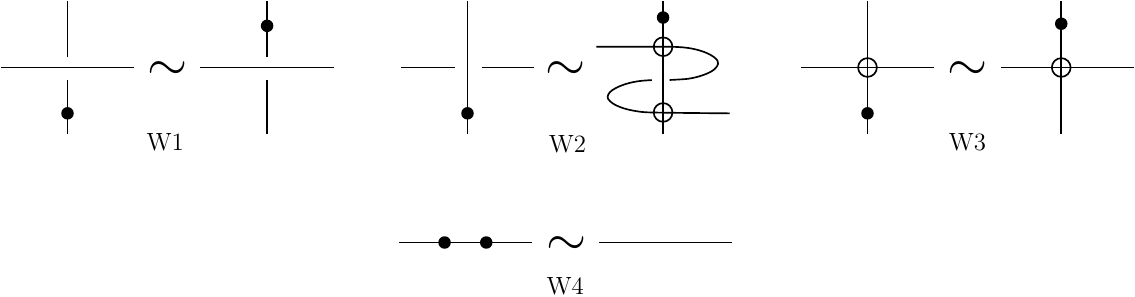}
\caption{Wen moves W1--W4.}
\label{fig101}
\end{figure}

The parity of the number of wens 
is an invariant of an extended welded knot. 
An extended welded knot is said to 
be \emph{of odd type} or \emph{of even type}
if it is presented by a virtual knot diagram with 
an odd or even number of wens, respectively. 
The inclusion maps 
\[
f:\mathcal{D}_0\rightarrow\bigsqcup_{n=0}^{\infty}\mathcal{D}_{2n} 
\mbox{ and }
g:\mathcal{D}_1\rightarrow\bigsqcup_{n=0}^{\infty}\mathcal{D}_{2n+1}
\]
induce the maps 
\[
\begin{array}{l}
f_*:\{\mbox{welded knots}\}\rightarrow
\{\mbox{extended welded knots of even type}\}
\mbox{ and}\\
g_*:\{\mbox{wen knots}\}\rightarrow
\{\mbox{extended welded knots of odd type}\}
\end{array}\]
naturally by taking the quotient 
under suitable Reidemeister moves and wen moves. 
Since any virtual knot diagram with a finite number of wens 
is related to a diagram with at most one wen by 
wen moves \textrm{W1}--\textrm{W4}, 
we see that $f_*$ and $g_*$ are surjective.

In this paper, we first prove that the map $g_*$ is injective; 
namely, we have the following.

\begin{theorem}\label{thm11}
Let $D$ and $D'$ be virtual knot diagrams with a single wen. 
If $D$ is related to $D'$ by a finite sequence of Reidemeister moves
\textrm{R1--R8} and wen moves \textrm{W1--W4}, 
then they are related by a finite sequence of  
\textrm{R1--R8} and \textrm{W1--W3}, without the need of \textrm{W4}.
\end{theorem}

Next we consider the map $f_*$. 
A \emph{horizontal mirror reflection} 
of a virtual knot diagram $D$ is obtained by 
reflecting it with respect to a line in the plane on which the diagram lies, 
as in Figure~\ref{F:MirrorImage}. 
We denote by $D^\dagger$ the obtained diagram, 
and the move from $D$ to $D^\dagger$ is labeled by 
\textrm{M}. 
Then we have the following.

\begin{figure}[htb]
\centering
\includegraphics[scale=0.6]{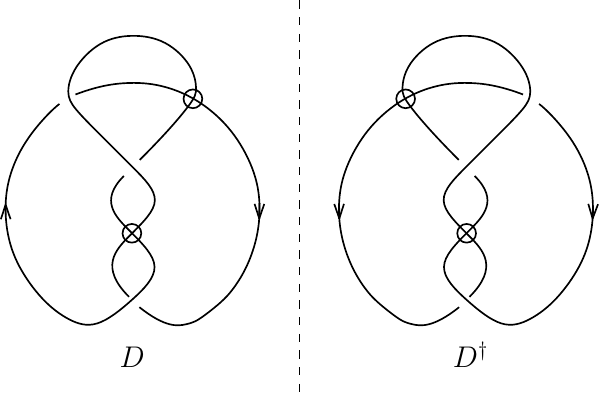}
\caption{A virtual knot diagram $D$ and its horizontal mirror reflection $D^\dagger$.}
\label{F:MirrorImage}
\end{figure}

\begin{theorem}\label{thm12}
Let $D$ and $D'$ be two virtual knot diagrams without wens. 
If $D$ is related to $D'$ by a finite sequence of Reidemeister moves
\textrm{R1--R8} and wen moves \textrm{W1--W4}, 
then 
they are related by a finite sequence of 
\textrm{R1--R8} and \textrm{M}, 
without the need of \textrm{W1--W4}.
\end{theorem}

Theorem~\ref{thm11} implies that 
the set of extended welded knots of odd type 
can be identified with that of wen knots (see also~\cite[Proposition~3.3]{D2}). 
Also, 
Theorem~\ref{thm12} induces an identification between the set of extended welded knots of even type 
and that of welded knots 
up to~\textrm{M}, improving \cite[Proposition~5.1]{D2}.
Therefore we have the following.

\begin{corollary}
\label{cor13}
There are one-to-one correspondences 
    
\begin{eqnarray*}
\{\mbox{extended welded knots of odd type}\}
&\stackrel{1:1}{\longleftrightarrow}&
\{\mbox{wen knots}\} \mbox{ and }\\
\{\mbox{extended welded knots of even type}\}
&\stackrel{1:1}{\longleftrightarrow}&
\{\mbox{welded knots}\}
/\textrm{M}.
\end{eqnarray*}
\end{corollary}

The proofs of Theorems~\ref{thm11} and~\ref{thm12} 
are given in Sections~\ref{sec2} and~\ref{sec3}, respectively. We remark that we will translate the problem in terms of Gauss diagrams to prove these results. 
In Section~\ref{sec4}, 
we study extended welded links, meaning extended welded knots with multiple components to describe their structure in terms of welded links and wen links.

\section{Extended welded knots of odd type}\label{sec2}

Instead of working with virtual knot diagrams with a finite number of wens, 
it is convenient to use their associated Gauss diagrams. These diagrams allow us to clearly summarise all the combinatorial data of the considered objects, making proofs more straightforward. 
Let $D$ be a virtual knot diagram with a finite number of wens. 
We regard $D$ as the image of a circle $C$ under the immersion described as follows.
For each real crossing of~$D$, 
we connect the pair of points that is the preimage of the crossing by a chord 
which is oriented from the over-crossing to the under-crossing, 
and we decorate the chord with the sign of the crossing.
The dots on $C$ that are preimages of the wens of $D$ are also called wens. The data of~$C$, the points on~$C$, and the signed oriented chords compose the Gauss diagram $G$ related to~$D$.
See Figure~\ref{F:ExampleGauss} for an example. 
\begin{figure}[htb]
\centering
\includegraphics[scale=0.5]{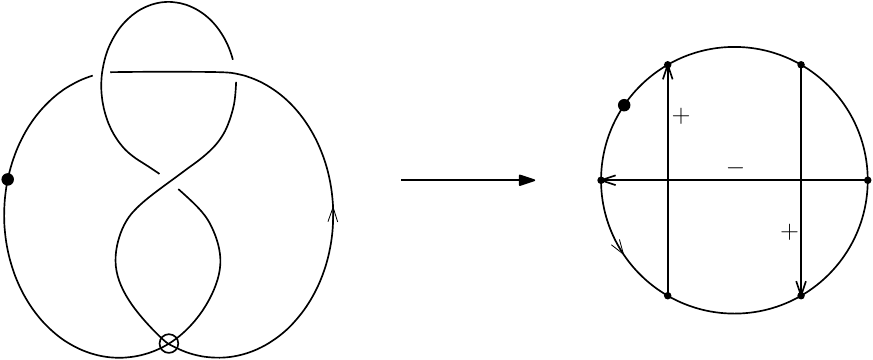}
\caption{A Gauss diagram associated to a wen knot diagram.}
\label{F:ExampleGauss}
\end{figure}

An extended welded knot can be represented as 
an equivalence class of such Gauss diagrams 
under translation of Reidemeister moves \textrm{R1}--\textrm{R3}, \textrm{R8}, and wen moves~\textrm{W1},~\textrm{W2}, and \textrm{W4} expressed in terms of Gauss diagrams. 
In fact, two virtual knot diagrams define the same Gauss diagram 
if and only if they are related by a finite sequence of 
\textrm{R4}--\textrm{R7} and~\textrm{W3}. 
Figure~\ref{fig201} shows 
wen moves \textrm{W1} and \textrm{W2} on Gauss diagrams. 
The horizontal mirror reflection \textrm{M} 
induces the change of the signs of all chords of $G$~\cite[Section~2.2]{Ichimori-Kanenobu:2012}.

\begin{figure}[htb]
\centering
\includegraphics{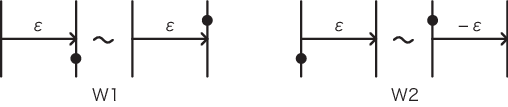}
\caption{Wen moves \textrm{W1} and \textrm{W2} on Gauss diagrams}
\label{fig201}
\end{figure}

Assume that a Gauss diagram $G$ has an odd number of wens. 
For a wen $w$ of $G$, we define the Gauss diagram $G(w)$ 
with a single wen, labeled $w$ again, as follows. 
Let $w=w_0,w_1,\dots,w_{2n}$ be the list of wens of $G$ 
in the order in which they appear on $C$ 
starting from $w=w_0$ and following the orientation of $C$. 
Let $A$ be the union of arcs on $C$ 
from $w_{2i-1}$ to $w_{2i}$ for $1\leq i\leq n$. 
The Gauss diagram $G(w)$ is obtained form $G$ 
such that 
\begin{enumerate}[label=(\roman*)]
\item the set of chords of $G$ and $G(w)$ are the same 
except their signs, 
\item \label{i:endopoint} if the initial endpoint of a chord of $G$ belongs to $A$, 
then we change the sign of the chord in $G(w)$, 
\item if the initial endpoint of a chord of $G$ belongs to $C\setminus A$, 
then the signs of the chord are the same in $G$ and $G(w)$, 
and 
\item we remove the wens $w_1,w_2,\dots,w_{2n}$ from $G$ 
and leave $w=w_0$ in $G(w)$. 
\end{enumerate}
See Figure~\ref{fig202} for an example, 
where the chords with circled signs satisfy condition~\ref{i:endopoint}.

\begin{figure}[htb]
\centering
\includegraphics{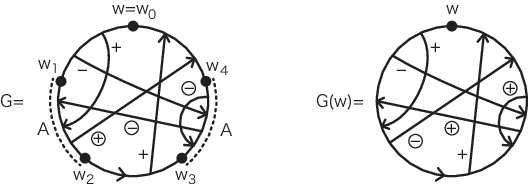}
\caption{Gauss diagrams $G$ and $G(w)$}
\label{fig202}
\end{figure}

\begin{lemma}\label{lem21}
Let $G$ be a Gauss diagram with odd number of wens. 
Then for any wens $w$ and $w'$ of $G$, 
the Gauss diagrams $G(w)$ and $G(w')$ with a single wen 
are related by a finite sequence of 
\textrm{W1} and \textrm{W2}. 
\end{lemma}

\begin{proof} 
Let $w_0,w_1,\dots,w_{2n}$ be the wens of $G$ in this order in
which they appear with respect to the orientation of~$C$. 
It is sufficient to prove the case where 
$w=w_0$ and $w'=w_1$. 
By definition, we see that 
$G(w')$ is obtained from $G(w)$ 
by sliding $w$ to the position of $w'$ 
opposite to the orientation of $C$. 
This is realized by a sequence of \textrm{W1} and \textrm{W2} only. 
\end{proof}

\begin{lemma}\label{lem22}
Let $G$ and $G'$ be two Gauss diagrams 
with an odd number of wens. 
Suppose that $G'$ is obtained from $G$ 
by one of \textrm{R1}, \textrm{R2}, \textrm{R3}, \textrm{R8}, \textrm{W1}, 
\textrm{W2}, and \textrm{W4}. 
Let $w$ be a common wen of $G$ and $G'$. 
Then the Gauss diagrams 
$G(w)$ and $G'(w)$ with a single wen are related 
by a finite sequence of 
\textrm{R1}, \textrm{R2}, \textrm{R3}, \textrm{R8}, \textrm{W1}, and 
\textrm{W2}, without the need of \textrm{W4}. 
\end{lemma}

\begin{proof}
The proof descends almost straightforwardly from the definition. 
In fact, if $G$ and $G'$ are related by \textrm{R1}, \textrm{R2}, or \textrm{R8}, 
then so are $G(w)$ and $G'(w)$. 
If $G$ and $G'$ are related by \textrm{W1} or \textrm{W4}, 
then we have $G(w)=G'(w)$. 

Assume that $G$ and $G'$ are related by \textrm{W2}. 
If the wen in \textrm{W2} is $w$, 
then $G(w)$ and $G'(w)$ are related by 
\textrm{W2}. 
If the wen in \textrm{W2} is not $w$, 
then we have $G(w)=G'(w)$. 

Assume that $G$ and $G'$ are related by \textrm{R3}. 
It is sufficient to check the move as shown in Figure~\ref{fig203}; 
the other cases are described as a combination of 
this move and \textrm{R2} moves,
or its local horizontal mirror reflection. 
We label the three chords 
by $1$, $2$, and $3$ as in the figure.

\begin{figure}[htb]
\centering
\includegraphics{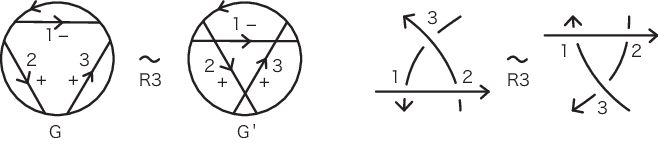}
\caption{A typical Reidemeister move \textrm{R3}}
\label{fig203}
\end{figure}

Comparing the signs of chords $1$, $2$, and $3$ 
in $G$ and $G(w)$, 
we have four cases as follows. 
We remark that the initial endpoints of chords 
$1$ and $2$ are adjacent on the circle $C$. 
\begin{enumerate}[label=(\roman*)]
\item \label{i:first} The signs of the chords $1$, $2$, and $3$ 
are the same in $G$ and $G(w)$, respectively. 
\item \label{i:second} The signs of the chords $1$ and $2$ 
are the same in $G$ and $G(w)$, respectively, 
and the sign of the chord $3$ is opposite. 
\item \label{i:third} The signs of the chords $1$ and $2$ 
are opposite in $G$ and $G(w)$, respectively, 
and the sign of the chord $3$ is the same. 
\item \label{i:fourth} The signs of the chords $1$, $2$, and $3$ 
are opposite in $G$ and $G(w)$, respectively.
\end{enumerate} 

In case \ref{i:first}, 
the Gauss diagrams $G(w)$ and $G'(w)$ are related by \textrm{R3}. 
In case \ref{i:second}, 
$G(w)$ and $G'(w)$ are related by a sequence of 
\textrm{R3} and \textrm{R8} moves as shown in Figure~\ref{fig204}.

\begin{figure}[htb]
\centering
\includegraphics{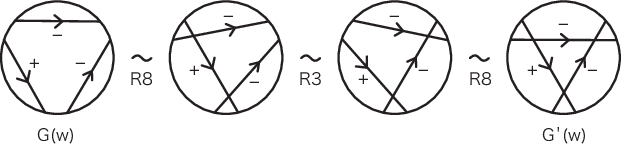}
\caption{$G(w)$ and $G'(w)$ are related by \textrm{R3} and \textrm{R8}}
\label{fig204}
\end{figure}

Cases \ref{i:third} and \ref{i:fourth} reduce to 
\ref{i:second} and \ref{i:first} by a local horizontal mirror reflection, 
respectively. 
Thus $G(w)$ and $G'(w)$ are related by a sequence of moves that do not include~\textrm{W4}. 
\end{proof}

The following is an interpretation of 
Theorem~\ref{thm11} in terms of 
Gauss diagrams. 

\begin{theorem}\label{thm23}
Let $G$ and $G'$ be two Gauss diagrams 
with a single wen.
If $G$ is related to $G'$ 
by a finite sequence of 
\textrm{R1}--\textrm{R3}, \textrm{R8}, 
\textrm{W1}, \textrm{W2}, and \textrm{W4}, 
then they are related by a finite sequence of 
\textrm{R1}--\textrm{R3}, \textrm{R8}, 
\textrm{W1}, and \textrm{W2}, 
without the need of \textrm{W4}.
\end{theorem}

\begin{proof}
Let $G=G_0,G_1,\dots,G_s=G'$ be 
a finite sequence of Gauss diagrams 
such that $G_{i+1}$ is obtained from $G_i$ 
by one of \textrm{R1}--\textrm{R3}, \textrm{R8}, 
\textrm{W1}, \textrm{W2}, and \textrm{W4}. 
For each $i$ with $0\leq i\leq s-1$, 
we chose a common wen $w_i$ of $G_i$ and $G_{i+1}$. 

Now we consider the sequence of Gauss diagrams 
with a single wen 
\[G_0(w_0),G_1(w_0),G_1(w_1),G_2(w_1),G_2(w_2),
G_3(w_2),\dots,
G_{s-1}(w_{s-1}),G_s(w_{s-1}).\]
Since $G$ and $G'$ have each a single wen, 
we have $G=G_0(w_0)$ and $G_s(w_{s-1})=G'$. 
Furthermore, for each $i$ with $0\leq i\leq s-1$, Gauss diagrams $G_{i+1}(w_i)$ and $G_{i+1}(w_{i+1})$ 
are related without \textrm{W4} moves by Lemma~\ref{lem21}, 
and $G_i(w_i)$ and $G_{i+1}(w_i)$ are related without \textrm{W4} moves
by Lemma~\ref{lem22}. 
Thus we have the conclusion. 
\end{proof}

\section{Extended welded knots of even type}\label{sec3}

In this section, 
we assume that $G$ is a Gauss diagram with an even number of wens. 
We denote by $G^\dagger$ the Gauss diagram 
obtained from $G$ by a horizontal mirror reflection~\textrm{M}. 
that is, $G^\dagger$ is obtained from $G$ 
by changing the signs of all chords of $G$.

\begin{lemma}\label{lem31}
Let $G$ be a Gauss diagram with no wen. 
Then $G$ and its horizontal mirror reflection 
$G^\dagger$ are related by 
a finite sequence of moves \textrm{W1}, \textrm{W2}, 
and \textrm{W4}.  
\end{lemma}

\begin{proof} 
We introduce a pair of wens 
on the circle $C$ of $G$ by \textrm{W4}. 
Then we move one of the wens around $C$ 
by \textrm{W1} and \textrm{W2}. 
We finally  cancel the pair of wens 
by \textrm{W4} again. 
Then the obtained diagram is $G^\dagger$. 
\end{proof}

Let $w_1,\dots,w_{2n}$ 
be the wens of $G$ appearing 
in this order along $C$. 
Let $A'$ (resp. $A''$) 
be the union of arcs on $C$ 
from $w_{2i-1}$ to $w_{2i}$ 
(resp. $w_{2i}$ to $w_{2i+1}$) 
for $1\leq i\leq n$, 
where $w_{2n+1}=w_1$. 
In convenience, if $G$ has no wen, 
then we put $A'=\emptyset$ and $A''=C$.

For a union of arcs $A\in\{A',A''\}$, 
we define the Gauss diagram $G(A)$ 
in a similar way to 
$G(w)$ in Section~\ref{sec2} as follows. 
The Gauss diagram $G(A)$ is obtained from $G$ 
in such a way that 
\begin{enumerate}[label=(\roman*)]
\item the set of chords of $G$ and $G(A)$ 
are the same except for their signs, 
\item if the initial endpoint of a chord of $G$ 
belongs to $A$, 
then we change the sign of the chord in $G(A)$, 
\item if the initial endpoint of a chord of $G$ 
belongs to $C\setminus A$, 
then the signs of the chord are the same 
in $G$ and $G(A)$, and 
\item we remove all the wens from $G$. 
\end{enumerate}

Let $G$ and $G'$ be two Gauss diagrams 
with an even number of wens. 
Assume that $G'$ is obtained from $G$ 
by one of \textrm{R1}--3, \textrm{R8}, \textrm{W1}, \textrm{W2}, or~\textrm{W4}. 
Fix a union of arcs $A$ for $G$. 
Then there is a union of arcs $A'$ for $G'$ 
such that $A$ and $A'$ coincide 
outside of the local move.

\begin{lemma}\label{lem32}
Let $G$, $G'$, $A$ and $A'$ be as above. 
Then the Gauss diagrams 
$G(A)$ and $G'(A')$ with no wen are related 
by a finite sequence of 
\textrm{R1}--\textrm{R3} and \textrm{R8}. 
\end{lemma}

\begin{proof}
The proof is almost the same as in Lemma~\ref{lem22} 
except for move \textrm{W2}. 
If $G$ and $G'$ are related by~\textrm{W2}, 
then we have $G(A)=G'(A')$ only. 
\end{proof}

The following is an interpretation of 
Theorem~\ref{thm12} in terms of Gauss diagrams. 

\begin{theorem}\label{thm33}
Let $G$ and $G'$ be two Gauss diagrams with no wen. 
If $G$ and $G'$ are related by a finite sequence of 
\textrm{R1}--\textrm{R3}, \textrm{R8}, \textrm{W1}, \textrm{W2}, 
and~\textrm{W4}, 
then they are related by a finite sequence 
of \textrm{R1}--\textrm{R3},~\textrm{R8}, and~\textrm{M}. 
\end{theorem}

\begin{proof} 
Let $G=G_0,G_1,\dots,G_s=G'$ be 
a finite sequence of Gauss diagrams 
such that $G_{i+1}$ is obtained from $G_i$ 
by one of \textrm{R1}--\textrm{R3}, 
\textrm{R8}, \textrm{W1}, \textrm{W2}, 
and \textrm{W4}. 
Put $A_0=\emptyset$ for $G_0$. 
We define the union of arcs $A_i$ for $G_i$ 
$(1\leq i\leq s)$ such that 
$A_{i-1}$ and $A_i$ 
coincide outside of the local move.

Now we consider the sequence of Gauss diagrams with no wens
\[G_0(A_0),G_1(A_1),G_2(A_2),
\dots,
G_{s-1}(A_{s-1}),G_s(A_s).\]
Since $G_0$ has no wen and $A_0=\emptyset$, 
we have $G_0(A_0)=G$. 
On the other hand, 
since $G_s$ has no wen and 
$A_s=\emptyset$ or $C$, 
we have $G_s(A_s)=G'$ or $G'^\dagger$. 
Furthermore, for each $i$ with $0\leq i\leq s-1$, 
the Gauss diagrams 
$G_i(A_i)$ and $G_{i+1}(A_{i+1})$ 
with no wen are related by 
a finite sequence of 
\textrm{R1}--\textrm{R3} and \textrm{R8} 
by Lemma~\ref{lem32}. 
By adding 
\[G_s(A_s)=G_s(C)=G'^{\dagger}
\stackrel{\textrm{M}}{\longrightarrow}
G'\]
after the sequence as above if necessary, 
we have the conclusion. 
\end{proof}

\begin{remark}
For a Gauss diagram $G$ 
of an extended welded knot of 
\textit{ odd} type, 
we can also consider a horizontal mirror reflection 
\textrm{M} as well as that of even type. 
In this case, 
\textrm{M} is generated by 
\textrm{W1} and \textrm{W2} 
so that we do not require it. 
\end{remark}

\section{Extended welded links}\label{sec4}

It is natural to generalize the notion of an extended welded knot 
to the case of links.
A $\mu$-component \emph{extended welded link} is 
an equivalence class of 
virtual link diagrams consisting of $\mu$ circles 
with a finite number of wens 
under \textrm{R1}--\textrm{R8} and \textrm{W1}--\textrm{W4}. 

We assume that a $\mu$-component 
extended welded links is \textit{ ordered}; that is, 
the components are labeled by $1,2,\dots,\mu$. 
We say that an extended welded link is 
\textit{ of type $(\delta_1,\delta_2,\dots,\delta_{\mu})$} 
if the $i$th component has even 
(resp. odd) number of wens 
for $\delta_i=0$ (resp. $\delta_i=1$). 
As well as an extended welded knot, 
it is convenient to use the Gauss diagram associated with 
a virtual link diagram. 

For each $i$ with $\delta_i=0$, 
let $\textrm{M}_i$ denote the operation for 
a Gauss diagram $G$ which changes 
the signs of chords whose initial endpoints 
belong to the $i$th component circle. 
Then we have the following. 
The proof is similar to those of 
Theorems~\ref{thm23} and \ref{thm33}, 
we leave it to the reader. 

\begin{theorem}\label{thm41}
Let $G$ and $G'$ be two Gauss diagrams 
of the same type $(\delta_1,\dots,\delta_{\mu})$. 
Suppose that the $i$th component has 
exactly $\delta_i$ wen for $i=1,2,\dots,\mu$. 
If $G$ and $G'$ are related 
by a finite sequence of 
\textrm{R1}--\textrm{R3}, \textrm{R8}, 
\textrm{W1}, \textrm{W2}, and \textrm{W4}, 
then they are related by 
a finite sequence of 
\textrm{R1}--\textrm{R3}, 
\textrm{R8}, \textrm{W1}, \textrm{W2}, and $\textrm{M}_i$ 
with $\delta_i=0$. 
\hfill$\Box$
\end{theorem}


\bibliography{WEN.bib}{}
\bibliographystyle{alpha}

\end{document}